\documentclass[12pt]{amsart}
\usepackage{amssymb,amsthm,amsmath,listings}
\usepackage[headings]{fullpage}
\usepackage{amsfonts,graphicx,transparent}
\usepackage{mdframed,xcolor,paracol}
\usepackage{color,bm}

\usepackage{multicol}
\newcommand{\twocols}{\begin{multicols}{2}}
\newcommand{\onecol}{\end{multicols}}
\newcommand{\threecols}{\begin{multicols}{3}}
\newcommand{\fourcols}{\begin{multicols}{4}}

\usepackage[all]{xy}
\usepackage[normalem]{ulem}
\usepackage[colorlinks,pagebackref=false]{hyperref}

\newcommand{\ZZ}{\mathbb{Z}}

\newcommand{\QQ}{\mathbb{Q}}

\usepackage{framed}
\newtheorem{Theorem}{Theorem}
\newtheorem{Program}[Theorem]{Program} 
\newtheorem{Corollary}[Theorem]{Corollary} 
\newtheorem{Lemma}[Theorem]{Lemma} \newtheorem{Proposition}[Theorem]{Proposition} \theoremstyle{definition}
  \newtheorem{Example}[Theorem]{Example}  \newtheorem{Remark}[Theorem]{Remark}  \newtheorem{Question}[Theorem]{Question}  
   
   \numberwithin{equation}{section}   
 \newtheorem*{Conjecture*}{Main Conjectures}

\DeclareMathOperator{\codim}{codim}

\def\<{\langle}
\def\>{\rangle}
\title{Arithmetic in the Boij-S\"oderberg Cone}
\author{Adam Boocher}
\address{Department of Mathematics, University of San Diego, San Diego, CA 92110, USA}
\email{aboocher@sandiego.edu}

\author{Noah Huang}
\address{Department of Mathematics, University of San Diego, San Diego, CA 92110, USA}
\email{nhuang@sandiego.edu}

\author{Harrison Wolf}
\address{Department of Mathematics, University of San Diego, San Diego, CA 92110, USA}
\email{harrisonwolf@sandiego.edu}

\begin{document}
\maketitle 
\vspace{-2.5em}
\begin{abstract}
We study two long-standing conjectures concerning lower bounds for the betti numbers of a graded module over a polynomial ring.  We prove new cases of these conjectures in codimensions five and six by reframing the conjectures as arithmetic problems in the Boij-S\"oderberg cone.  In this setting, potential counterexamples correspond to explicit Diophantine obstructions arising from the numerics of pure resolutions.  Using number-theoretic methods, we completely classify these obstructions in the codimension three case revealing some delicate connections between betti tables, commutative algebra and classical Diophantine equations.  The new results in codimensions five and six concern Gorenstein algebras where a study of the variety determined by these Diophantine equations is sufficient to resolve the conjecture in this case. 
\end{abstract}

\section{Introduction}
\noindent Let $R = k[x_1, \ldots, x_n]$ be a polynomial ring over a field $k$ and let $M$ be a finitely-generated graded $R$-module.  Two long-standing conjectures concern the {\bf betti numbers} of $M$ and their relationship to  $c$, the codimension of $M$ (the height of the annihilator of $M$).
\begin{Conjecture*} Let $R$ and $M$ be as above and let $c = \codim M$. Then 
\begin{enumerate}
\item 	$\beta_i(M) \geq {c\choose i}$,
\smallskip
\item If $M$ is not (isomorphic to) a complete intersection then $\sum \beta_i(M) \geq 2^c+ 2^{c-1}$.
\end{enumerate}\
\end{Conjecture*}
\noindent We give a brief history of these conjectures in Section 2, but mention now that they are open when $c\geq 5$.  
The first main result of this paper establishes new cases of these conjectures with the assumption that $M = R/I$ defines a Gorenstein ring.

\begin{Theorem}\label{Thm:GorCodim5}
Suppose that $M = R/I$ is a Gorenstein ring of codimension $c$, not isomorphic to a complete intersection. \begin{itemize} \item If $c = 6$ then Conjecture $(1)$ holds for $i = 0,1,2,4,5,6$ without any further assumption on $M$. 
  \item With the added assumption that $M$ is pure, Conjecture (2) is true if $c\in \{5,6\}$ and Conjecture (1) is true if $c= 5$. 
  \end{itemize}
\end{Theorem}
A module $M$ is called {\bf pure} if it is Cohen-Macaulay and for all $i\geq 0$ its $i$th syzygy module is generated in a single degree $d_i$. To such a module we associate its {\bf degree sequence} $\{0,d_1, \ldots, d_c\}\footnote{Without loss of generality, we take $d_0=0$}.$  By a result of \cite{HK}, there is (almost) a formula for the betti numbers of such modules (the so-called Herzog-K\"uhl Equations \ref{eqn:herzog-kuhl}). Indeed, for each $D$ there is a list of integers $B(D)$ such that if $M$ is pure with degree sequence $D$ then its list of betti numbers, $\beta(M)$ satisfies $\beta(M) = LB(D)$ for some integer $L$. 

\begin{Example}\label{Ex:1} Suppose $M$ is a pure module with $c=3$ and degree sequence $D$.
\begin{enumerate} 
\item If $D = \{0,2,3,8\}$ then $B(D) = \{5, 20, 16, 1\}$ and thus $\beta(M)=\{5L, 20L,16L,L\}$. 
\item If $D= \{0,1,3,4\}$ then $B(D) = \{1,2,2,1\}$ and thus $\beta(M) = \{L, 2L, 2L, L\}$.  
\end{enumerate}
\end{Example}
\noindent Notice that the numerics of the two examples are fundamentally different.  For the first, regardless of what $L$ is, $B(D)$ is ``large'' enough to ensure that $M$ satisfies both conjectures.  On the other hand, in the second example, $B(D)$ itself violates the bounds of both conjectures.  When this happens we say that $D$ is an {\bf arithmetic obstruction} to the conjecture.  In other words, if $L=1$ then $\beta(M) = B(D)$ would violate the conjectures and so to prove that such a module $M$ does indeed satisfy the conjecture, one must show that it is impossible for $L$ to be $1$.  This is simple enough in this case, since such a module would have the minimal free that begins $R^2 \to R^1\to M$ and then evidently $M \cong R/I$ where $I$ is an ideal generated by $2$ minimal generators, which is impossible for an ideal of height $c=3$ by Krull's altitude theorem.  

In principle, one approach to resolving the Main Conjectures would be to classify all arithmetic obstructions and then rule out the corresponding small values of $L$.   In fact, one can get away without classifying \emph{all} obstructions (such an analysis in codimensions 5 and 6, is precisely how we prove Theorem \ref{Thm:GorCodim5}).  At the same time it \emph{is} possible to completely classify all obstructions in codimension 3:
\begin{Theorem}\label{thm1}
If $c=3$ then \begin{itemize}\item There are {\bf infinitely many} arithmetic obstructions to Main Conjecture (1). These are described explicitly in Theorem \ref{inf family with recursion} and all have either $B_1=2$ or $B_2=2$.
 \item 	
Regarding Main Conjecture (2), there are {\bf exactly 4} degree sequences $D$ with \\ $\sum B(D) < 12$. These are described explicitly in Theorem \ref{theorem for equal to 10}.
 \end{itemize}
If $M$ is a pure module that is not a complete intersection and its degree sequence $D$ is one of the offending sequences above, then in the equation $\beta(M) = L B(D)$, $L$ must be at least 2, recovering that the Main Conjectures hold in this case.
\end{Theorem}

\noindent Taken together, the above Theorems give a proof of concept that this framework can shed light on questions concerning the ranks of syzygies. \medskip

From the point of view of Boij-S\"{o}derberg Theory (see e.g. \cite{BS,EFW, ES}), the set of all betti tables of Cohen-Macaulay modules of codimension $c$ lies in a simplicial cone whose extremal rays correspond to pure modules.  If $\rho_D$ is the extremal ray corresponding to degree sequence $D$ then, $B(D)$ gives the coordinates of the smallest integer point on that ray, and the theorem of Herzog-K\"uhl asserts that every betti table of any pure module of type $D$ will lie on $\rho_D$ as an integer multiple of $B(D)$. Our approach is similar in spirit to the one initiated in  \cite{ErmanSemigroup}, because although calculating $B(D)$ is immediate it can be quite difficult to tell whether or not the numbers $B(D)$ are indeed the betti numbers of a module. 

Because the numerics of the Boij S\"oderberg cone are rather subtle, we find pure modules to be an excellent testing ground for the Main Conjectures.  When $1\cdot B(D)$ obeys the bounds, that provides interesting positive evidence for the conjectures and when not, it directs our attempt to understand for which $L$, $L\cdot B(D)$ can be the betti sequence of a module.  

In an effort to systematically study the general case, we propose the following. 

\begin{Program}\label{Program:1}
Let $c$ be a positive integer.
\begin{enumerate}
	\item Identify those degree sequences $D = \{0, d_1, \ldots, d_c\}$ with the property that $B(D)$ violates Main Conjecture (1) or (2). We will call these {\bf arithmetic obstructions}.
	\item Recalling that if $M$ is a pure module of type $D$ then $\beta(M) = LB(D)$, for some integer $L \geq 1$, prove that if $M$ is a pure module with betti sequence $L\beta(D)$, then $L$ must be sufficiently large so that $\beta(M)$ satisfies the conjecture violated by $B(D)$.
\end{enumerate}
\end{Program}
	
\noindent When $c \in \{1, 2\}$ this program is trivial and Theorem \ref{thm1} settles the case when $c=3$.  For $c\geq 4$ this program seems quite difficult (e.g. see the discussion in Section \ref{subsection3.1}) and we hope that techniques from neighboring fields might help.



\bigskip 

\noindent{{\bf Why is this interesting?}  If one views Theorem \ref{thm1} as a new way of verifying a special case of the Main Conjectures when $c = 3$, then this is perhaps not so interesting. After all, both conjectures have been long known to be true for \emph{any} module in codimension $c\leq 4$. However, there are many features of Program \ref{Program:1} that are remarkable and and worth more careful study. 

\begin{itemize}
\item Although we cannot complete Program \ref{Program:1} in full generality for $c\geq 4$, this general approach is strong enough to resolve the Main Conjectures in several new cases for codimension 5 and 6, (see Theorem \ref{Thm:GorCodim5}). Indeed, the last section of the paper illustrates that resolving the Main Conjectures can be done even if fully classifying the obstructions is impossible. 
\item The classification in step (1) requires solving (systems of) diophantine equations which is a notoriously delicate enterprise.  As $c$ grows, it is not clear at all how the solutions to these systems will behave.

For instance when $c = 2$ we must find integer solutions to the equation
$$yz = 2(y-x)(z-x).$$
We solve this by first finding the rational solutions. This step can be done for any $c$ (see Theorem \ref{thm:generalprojectionproof}) but finding the integer solutions relies on some miracles that occur in the proof that then reduce the problem to solving Pell's equation, which has infinitely many solutions.  

For $c=3$ one of many equations that would need to be studied is 
$$yzw = 3(y-x)(z-x)(w-x).$$
We do not know whether there are infinitely many solutions $0< x< y< z< w$ with $\gcd(x,y,z,w)=1$.  Although not exactly relevant to our situation, we point out the fact that many elliptic curves have infinitely many rational points, but they always possess a finite number of integer points.  Do the systems of diophantine equations here follow a similar yoga?  
\item Although the number theory involved in the $c=3$ case is classical, and uses nothing more than Pell's Equation, to complete step (2) of the program one must use some rather sophisticated commutative algebra.  To verify that $\beta_1, \beta_2 \geq 3$ one can appeal to the principal ideal theorem.  However, regarding Main Conjecture (2), the degree sequence $D = \{0,3,5,8\}$ has $B(D) = \{1,4,4,1\}$, and $D' = \{0,1,5,6\}$ has $B(D') = \{2,3,3,2\}$ which both violate Conjecture 2.  To rule these out requires some more hard-hitting machinery, for instance $D$ can be addressed by the fact that an almost complete intersection is never Gorenstein \cite{Kunz}.  For $D'$ we could appeal either to the fact that any module $M$ with a $2\times 3$ linear presentation matrix must have quadratic Buchsbaum-Rim Syzygies, or perhaps by appealing to the fact that $\beta_1(M) \geq \beta_0(M) + (c-1)$. For both of these, see \cite{BERemarks}. 
Alternatively, one could use the Syzygy Theorem \cite{EG}. 
\end{itemize}
In short, our answer for why this is interesting is that we think the arithmetic questions deserve a careful study, and their resolution will then likely lead to interesting candidates for betti numbers that will need to be addressed using techniques from commutative algebra.  

Finally, we present some intriguing data from a computer search which generated following question. We call a degree sequence self-dual if its first difference is symmetric.   It is interesting to see which self-dual degree sequences could be the degree sequence of a Gorenstein ring that is not a complete intersection.  Evidently, $B_0(D)$ must equal $1$, and because an almost complete intersection is never Gornstein \cite{Kunz}, we need $B_1(D) \geq c+2$. 
\begin{Question}
For which $c$ does there exist a self-dual degree sequence $D$ with the property that $B(D)$ has $B_0(D) = 1$ and $B_1(D)\geq c+2$ that violates Main Conjecture (1) or (2)?  
\end{Question}
Computationally, we have found examples for $c = 15, 19, 25$. Such examples could suggest questions like: 

\begin{Question}\label{quest:gor}
Is there a Gorenstein algebra $R/I$ whose resolution is the following?
$$\begin{array}{ccccccccccccccccc}
       & 0 & 1 & 2 & 3 & 4 & 5 & 6 & 7 & 8 & 9 & 10 & 11 & 12 & 13 & 14 & 15\\
      \text{} & 1 & 85 & 630 & 2295 & 4998 & 6630 & 4590 & 2210 & 2210 & 4590 & 6630 & 4998 & 2295 & 630 & 85 & 1\\
      0: & 1 & . & . & . & . & . & . & . & . & . & . & . & . & . & . & .\\
      1: & . & 85 & 630 & 2295 & 4998 & 6630 & 4590 & . & . & . & . & . & . & . & . & .\\
      2: & . & . & . & . & . & . & . & 2210 & . & . & . & . & . & . & . & .\\
      3: & . & . & . & . & . & . & . & . & 2210 & . & . & . & . & . & . & .\\
      4: & . & . & . & . & . & . & . & . & . & 4590 & 6630 & 4998 & 2295 & 630 & 85 & .\\
      5: & . & . & . & . & . & . & . & . & . & . & . & . & . & . & . & 1
      \end{array}$$
Such an example would violate both Conjectures (1) and (2). This is the smallest symmetric degree sequence we have identified that cannot be ruled out using the techniques discussed above. If one doesn't require that $M$ be a cyclic module, then one can ask whether there is a Cohen-Macaulay module whose resolution is 
$$\begin{array}{cccccccccc}
       & 0 & 1 & 2 & 3 & 4 & 5 & 6 & 7 & 8\\
      \text{total:} & 4 & 25 & 60 & 60 & 42 & 60 & 60 & 25 & 4\\
      0: & 4 & 25 & 60 & 60 & . & . & . & . & .\\
      1: & . & . & . & . & 42 & . & . & . & .\\
      2: & . & . & . & . & . & 60 & 60 & 25 & 4
      \end{array}.$$
\end{Question}
\noindent This betti diagram appears in the survey article by the first author and Grifo \cite{BoocherGrifo}. Motivated by the seemingly different behavior for the two conjectures when $c=3$ we ask the following. 
\begin{Question}
Let $c \geq 4$.  
\begin{itemize}
\item Like in the codimension $3$ case, are there are infinitely many degree sequences of length $c$ that violate Main Conjecture (1)?  Does this answer change if we refine the question to also depend on the index $i$?
\item Similarly, are there only finitely many that violate Main Conjecture (2)?
\item What obstructions from Commutative Algebra are required to ensure that $L$ must be sufficiently large?
\end{itemize}
\end{Question}

\begin{Question}
	If it turns out that for some (most?) $c$ there are only finitely many arithmetic obstructions, for such $c$ let $\lambda_c$ denote the largest final index $d_c$ of any violator $D$. What are the asymptotics of the sets of numbers $\lambda_c$?
\end{Question}

\section{Background and History}\label{Sec:Back}
\noindent {\bf Brief history of the Main Conjectures:}  
For a detailed history of Conjectures (1) and (2) we refer the reader to \cite{BoocherGrifo}.  Here we give a brief summary with the goal of explaining the link between the two conjectures.  Main Conjecture (1) was first stated for cyclic modules $R/I$ independently by Horrocks \cite{HartshorneProblems} and Buchsbaum-Eisenbud \cite{BE}.  In the latter, their conjecture actually was that the minimal resolution of $R/I$ would support the structure of a DG-Algebra, and by comparison with a Koszul complex, the bound $\beta_i \geq {c \choose i}$ would follow.  This stronger statement about algebra structures was resolved in the negative by Avramov \cite{LuchoObstructions} but Main Conjecture (1) as stated here remains.  It is this form that is often called the Buchsbaum-Eisenbud-Horrocks Rank Conjecture.  It is known in many special cases, but is open in general if $c\geq 5$.

Until recently there was a ``middle'' conjecture, that $\sum \beta(M)\geq  2^c$, known as the Total Rank Conjecture.  This was settled by Walker \cite{W} in the case that the characteristic of $k$ is not $2$ and by VandeBogert-Walker \cite{VandeW} in the remaining case. As part of their work, they showed that equality holds if and only if $M$ is a complete intersection.  The content of Main Conjecture (2) is then that after $2^c$ the next possible total rank is $2^c+2^{c-1}$. 

When $c\leq 4$, the minimal possible betti numbers of $M$ are known. Here, by minimal we mean that the smallest elements in the poset determined by term-by-term inequality.  This classification appears in \cite{CEM} where they noted that the total sum was always at least $2^{c}+2^{c-1}$ except in the case that $M$ was isomorphic to a complete intersection, our Main Conjecture (2). Contemporaneously, the same statement was obtained for multigraded artinian modules in \cite{CE,Chara}.  Since then, Main Conjecture (2) has been verified in several other cases. Notably, in each of these results the bound $2^{c}+2^{c-1}$ appears in different guises.  For instance, in the multigraded Artinian case, it is true that if $R/I$ is not isomorphic to a complete intersection then $\beta_i(R/I) \geq {c\choose i} + {c-1 \choose i}$ from which the bound holds from summing over $i$.  This is not true in general (even for monomial ideals $I$ if $R/I$ is not Artinian!) but nevertheless a separate argument verifies Main Conjecture (2) in this general monomial case \cite{BS}.  Similarly, there is a rather strange behavior discussed in \cite{BoocherW} where if the module $M$ has suitably low regularity, then roughly the first half of the betti numbers $\beta_1, \beta_2, \ldots \beta_{c/2}$ satisfy $\beta_i \geq 2{c \choose i}$ whereas the other half obey $\beta_i \geq {c\choose i}$ by Earlier work of Erman \cite{Erman}. Thus, roughly, on \emph{average} each betti number is at least $1.5{c\choose i}$ and the sum is thus $1.5\cdot 2^c$ which is none other than $2^c+ 2^{c-1}$.  A classification of the betti tables in the Koszul almost complete intersection case by \cite{Mastroeni} also verifies Main Conjecture (2) as well. 

\medskip 

\noindent {\bf Relevant Boij-Sderberg Background:} Throughout this paper when we say {\bf pure module} we will mean a Cohen-Macaulay module $M$ whose $i$th syzygy module is generated in a single degree $d_i$.  If $c$ is the codimension of $M$ then $d_i = 0$ for $i>c$.  After shifting, if necessary, we will assume that $M$ is generated in degree $0$.  Thus to such a module $M$ we associate to it the {\bf degree sequence} $D = \{0, d_1, \ldots, d_c\}$, which is strictly increasing.  As a slight abuse of language, we will say that $D$ is a codimension $c$ degree sequence. By a result of \cite{HK} the betti numbers $\beta_i(M)$ for $i> 0$ are given by 
\begin{equation}\label{eqn:herzog-kuhl}\beta_i(M) = \beta_0(M)\prod_{i\neq j} \frac{d_j}{|d_i-d_j|}.\end{equation}

The numbers $\pi_i(D) := \prod_{i\neq j} \frac{d_j}{|d_i-d_j|}$ are thus crucial in understanding the betti numbers $\beta_i(M)$. (By definition, set $\pi_0 = 1$).  Note that because of the homogeneity in the formula for $\pi_i(D)$, if $\lambda \in \ZZ$, then $\pi_i(\lambda D) = \pi_i(D)$.  For this reason, we will be only interested in those $D$ with $\gcd(D) = 1$.  We call such $D$ non-degenerate, and unless otherwise stated, any statement that seeks to find all $D$ with a given property will seek to find the non-degenerate  ones.

We can think of the tuple $\pi(D) := (1, \pi_1(D), \ldots, \pi_c(D))$ as a vector in $\QQ^{c+1}$.  Then $\beta_i(M) = \beta_0(M)\cdot \pi_i(D)$.  We will sometimes refer to the rational numbers $\pi_i(D)$ but since at the end of the day, betti numbers are integers we also define $B(D)$ to be the smallest positive integer multiple of the vector $(1,\pi_i(D), \ldots, \pi_c(D))$ that lies in $\ZZ^{c+1}$, and will write $B_i(D)$ for its $i$th component with $i=0, \ldots, c$.  We note that what we call $B(D)$ is what is produced by the Macualay2 software package BoijS\"oderberg by calling \texttt{pureBetti D}.   Then it follows that if $M$ is any pure module with degree sequence $D$ then $\beta_i(M) = L\cdot B(M)$ for some positive integer $L$.  Note that $L$ depends on $M$.  We will often suppress this dependence because for our interests we will be considering all $M$ and seek a bound for $\beta_i(M)$, and so we will consider what the possible values of $L$ can be.

Finally, we remark that the dual $M^*$ of any pure module $M$ with degree sequence $D$ is also pure and has degree sequence$$D^* := \{0, d_c-d_{c-1}, d_c-d_{c-2}, \ldots, d_c -d_1, d_c\}.$$

\begin{Example}
If $D = \{0, 3, 4, 8\}$ then $\pi(D) = \{1, 32/5, 6, 3/5\}$, and $B(D) = \{5,32,30,3\}$.   Any pure module $M$ with this degree sequence will have betti numbers $\beta_i(M)= LB_i(D)$ for some $L$.  $D^* = \{0,4,5,8\}$ and has $B(D^*)=\{3,30,32,5\}$.  We show the betti diagrams below:
$$	\begin{array}{c} B(D)\\ \\ 
\begin{array}{c|cccc}
       & 0 & 1 & 2 & 3\\ \hline 
      0& 3 & . & . & .\\
      1 & . & . & . & .\\
      2 & . & . & . & .\\
      3 & . & 30 & 32 & .\\
      4 & . & . & . & .\\
      5 & . & . & . & 5
      \end{array} \end{array}
      \hskip 1in  
      \begin{array}{c}  B(D^*) \\ \\ 
      \begin{array}{c|cccc}
       & 0 & 1 & 2 & 3\\ \hline 
      0 & 5 & . & . & .\\
      1 & . & . & . & .\\
      2 & . & 32 & 30 & .\\
      3 & . & . & . & .\\
      4 & . & . & . & .\\
      5 & . & . & . & 3
      \end{array}
	\end{array}$$
\end{Example}

It was proven in \cite{ES,EFW} that pure modules exist for any degree sequence.  That is, for some $L$ there is a module $M$ with betti numbers $\beta_i(M) = LB_i(D)$. There are explicit constructions of such modules, but these constructions typically result in a rather large value of $L$ that is far from the minimal such one.  For instance, the Macaulay2 command \texttt{pureAll} can be used to find three explicit values for $\beta_0(M)$ that are possible.  In the previous example, these numbers are $35, 175, 10$ and so in particular we see that there is a module $M$ with betti numbers $\beta_i(M) = \{10, 64,60, 6\}$ or $L=2$ in our case.   However, in fact there actually is a module with $L=1$ which can be discovered by some experimentation with Macaulay2, using the commands \texttt{randomModule, randomSocleModule}.   For comparison, the codimension 15 degree sequence in Question \ref{quest:gor}, we are asking if $L=1$ is possible.  The bound supplied from known constructions is $L = 18240$.

Finally, because these two lemmas are useful in the rest of the paper, we include them here so that the remaining sections can flow more clearly. 

\begin{Lemma}\label{lemma:even} Suppose that $D$ is a degree sequence of arbitrary length $c$ and let $B = B(D)$.  Then 
\begin{enumerate}
	\item $B_0 - B_1 + \cdots + (-1)^c B_c = 0$.
	\item The sum of the terms with even index $s_{even}= \sum B_{2k}$ is equal to the sum of the terms with odd index $s_{odd} = \sum B_{2k+1}$
	\item The sum of all $B_i$ is even: $\sum B_i = 2s_{odd}$.
\end{enumerate}
\end{Lemma}
\begin{proof}
Parts (2) and (3) follow immediately from part (1), which can be verified, either by a calculation with the formulas for $\pi(D)$, or by appealing to the fact that if $M$ is a Cohen Macaulay module of codimension $M$.  Then if $\overline{M}$ is an Artinian reduction of $M$ then $\overline{M}$ will be of finite length and thus will have rank $0$ as a module.  Note that the betti numbers of $M$ and $\overline{M}$ coincide.  Then by the Herzog-K\"uhl equations we have
$$0 = \mathrm{rank}(M) = \sum_{i=0}^c (-1)^i \beta_i(\overline{M}) = \sum_{i=0}^c (-1)^i \beta_i(M) = \sum_{i=0}^c (-1)^i L\cdot B_i(D)$$
for some positive integer $L$.  The result follows upon division by $L$. 
\end{proof}

\begin{Proposition}\label{prop:Lis1ifB1istwo}
Suppose $D = \{0,d_1, \ldots, d_c\}$ is a degree sequence. Then 
\begin{enumerate}
\item $\pi_1(D) > 1$;
\item $B_1(D) > B_0(D)$;
\item If $B_1(D) = 2$ then $B(D) = \pi(D)$. 
\end{enumerate}
\end{Proposition}
\begin{proof}
Before clearing denominators we see that
$$\pi_1(D) = \frac{d_2\cdots d_c}{(d_2-d_1)\cdots(d_c-d_1)} = \frac{d_2}{d_2-d_1}\cdots\frac{d_c}{d_c-d_1} > 1\cdots 1 = 1$$
establishing $(1)$. Now if $\pi_1(D)$ is written in lowest terms $P/Q$ then by part (1) we have that $P > Q$.  Clearing denominators must require multiplying $\pi(D)$ be a multiple of $Q$, so we have that $Q|B_0(D)$, say $B_0(D) = QR$.  Then $B_1(D) = PR > QR = B_0(D)$ proving $(2)$.  Continuing with this notation, suppose now that $B_1(D) =2$.  Then since
$$2 = PR > QR \geq 1$$
the only possibility is for $Q = R = 1$ so $B_0(D) = 1$ and $B(D) = \pi(D)$, finishing the proof. 
\end{proof}

\section{Codimension 3 Degree Sequences with $B_1$ or $B_2$ equal to $2$}
\noindent  In this section we will suppose that $D = \{0,x,y,z\}$ is a degree sequence of codimension $3$. The goal of this section is to complete Program \ref{Program:1} for Main Conjecture (1) when $c=3$. In other words, we seek to find all $D$ so that $B(D)$ is not at least $\{1,3,3,1\}$.  After dualizing if necessary, the main content here is to understand when $B_1(D) = 2.$ By Proposition \ref{prop:Lis1ifB1istwo} we know that $B(D) = \pi(D)$. Hence our goal is to solve $$2 = \frac{yz}{(y-x)(z-x)},$$ or equivalently 
\begin{equation}\label{eq:1} 2(y-x)(z-x)=yz.\end{equation}
Let $x_0,y_0,z_0$ be integers.  Our goal is to determine when $(x_0,y_0,z_0)$ is a solution to Equation \ref{eq:1}.  Note that because of homogeneity, if $z_0 \geq 0$ then $(x_0,y_0, z_0)$ is a solution if and only if $(x_0/z_0,y_0/z_0,1)$ is a solution. Thus we look to find all rational solutions to the equation with $z = 1$:
\begin{equation}\label{eq:2} 2(y-x)(1-x)=y.\end{equation}

Note that the point $P=(0,0,1)$ is one such solution, and the only solution with $x = 0$.  If $Q=(X_0,Y_0,1)$ is another rational solution then the line in the $xy$ plane joining $P$ and $Q$ will be given by $y = (Y_0/X_0)x$ which has rational slope.  Conversely, any line with rational slope $m$ through $P$ will intersect the plane curve $C$ given by Equation \ref{eq:2} in at most one other point, $Q$ which must have rational coordinates $(X_0, Y_0, 1)$ with $Y_0/X_0 = m$.  
\begin{center}
\begin{figure}
\includegraphics[width=3in]{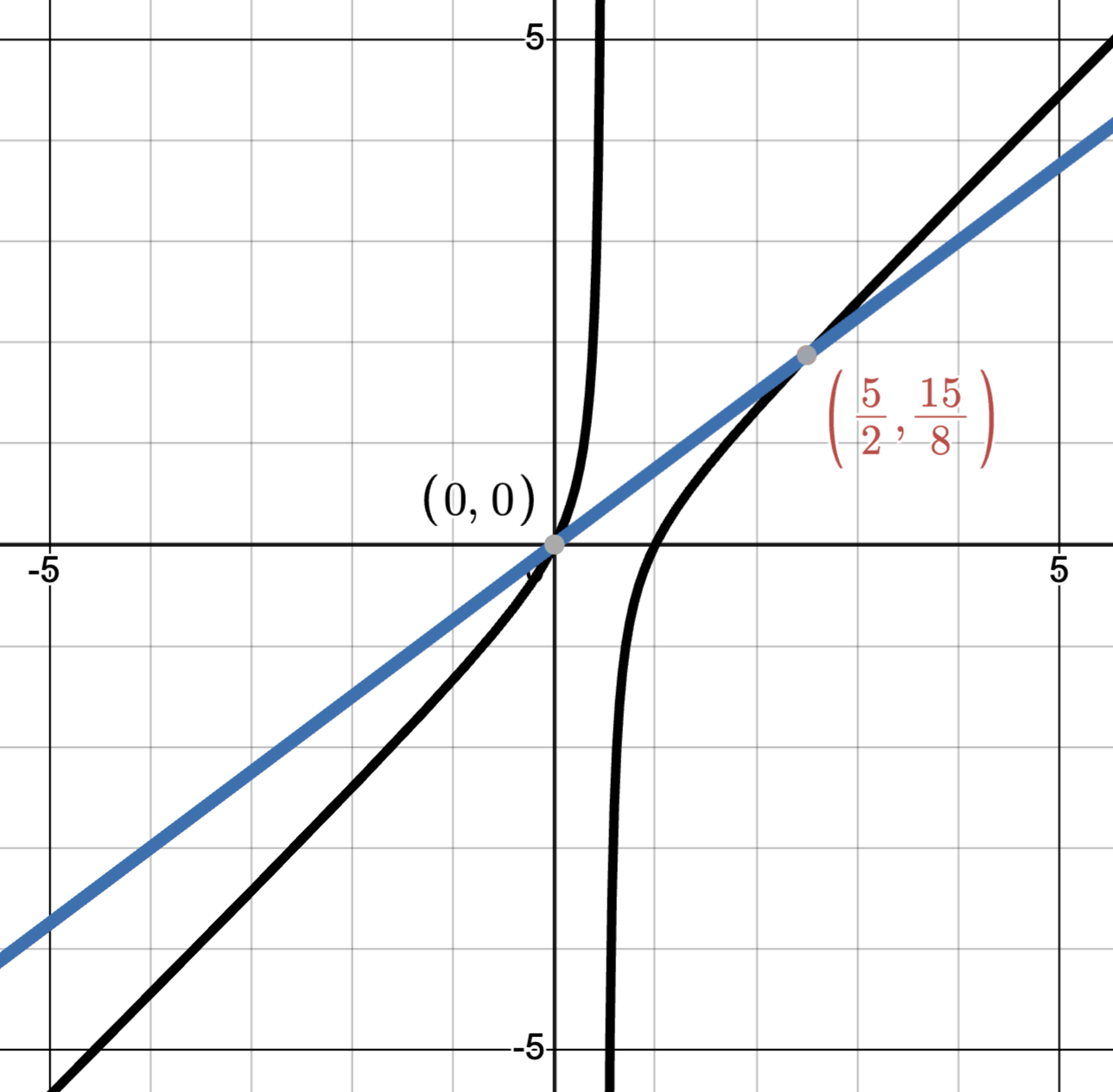}	
\caption{The graph of $2(y-x)(1-x)=y$. Rational points, which are parametrized by lines with rational slopes through $(0,0)$, correspond to integer solutions to equation \ref{eq:1}.}
\end{figure}
\end{center}
Solving for $Q$ means finding the intersection of the line with the curve $C$, which yields the equation $$2(mx-x)(1-x)=mx,$$ which simplifies to 
$$x((2-2m)x+m-2) = 0.$$
If $m = a/b$ (with $\gcd(a,b) = 1$) this yields that $Q$ has coordinates:
$$\left(1+\frac{a}{2(b-a)}, \frac{a}{b}+\frac{a^2}{2b(b-a)},1\right) =\left(1+\frac{a}{2(b-a)}, \frac{a(2b-a)}{2b(b-a)},1\right).$$ 
This description is a complete parametrization of all rational points on $C$. To determine integer solutions Equation \ref{eq:1} we have to find all multiple that will clear denominators.  To do so we need to first put these fractions in lowest terms.  

Note that $\gcd(a,b-a) = \gcd(a,b) = 1$ and $\gcd(2b-a,b-a) = \gcd(2b-a,b) = \gcd(a,b) = 1$, so if there are any common factors then they must be with the $2$ in the denominator.  In other words, either the fractions are in lowest terms or we may need to cancel a factor of $2$ and then it will be in lowest terms.  

{\bf Case 1}: $a$ is odd.  In this case, both fractions have odd numerator and so the fraction is in lowest terms.  

{\bf Case 2}: $a$ is even, say $a= 2A$.  Note since $\gcd(a,b)= 1$ we must have that $b$ is odd. Then after canceling the factor of $2$, we have the expression in lowest terms:
$$\left(1+\frac{A}{(b-2A)}, \frac{A(2b-2A)}{b(b-2A)},1\right).$$
Since these expressions involve fractions in lowest terms, we can clear denominators to classify the integer solutions.  This proves the following proposition.
\begin{Proposition}
	The integer solutions $(x_0, y_0, z_0)$ to Equation \ref{eq:1} with $z_0 \neq 0$ are the integer multiples of 

(Type 1) $\left(b(2b-a), a(2b-a),2b(b-a)\right)$, with $a$ odd and $\gcd(a,b) = 1$.

(Type 2) $\left(b^2-Ab, A(2b-2A),b(b-2A)\right)$, with $b$ odd and $\gcd(a,b) = 1$.
\end{Proposition}
The next result is critical.

\begin{Theorem}\label{Thm:AllarePell}
If $D = \{0,x_0, y_0, z_0\}$ is a degree sequence with $B_2(D) = 2$ then $z_0 = y_0 + 1$.	
\end{Theorem}\begin{proof}
\noindent We will use the fact that all of the entries of $B(D) =\pi(D)$ are integers, i.e. that $B_2 = \pi_2(D)$ must also be an integer, so we have that 
$$\frac{xz}{(y-x)(z-y)}\in \ZZ$$
for all of the solutions described above.   Inserting them into this, we obtain for Type 1: 
\begin{eqnarray*}
\frac{xz}{(y-x)(z-y)} &=& \frac{b(2b-a)2b(b-a)}{(a(2b-a)-b(2b-a))(2b(b-a)-a(2b-a))} \\ 
& = & \frac{-2b^2}{2b^2+a^2-4ab}.
\end{eqnarray*}
This fraction must be an integer, so the denominator must be a factor of the numerator. We will show that the denominator is in fact $\pm 1$. Since the denominator is odd (as $a$ is odd), it must be that $b^2$ is divisible by the denominator, so say $-2b^2 = \ell(2b^2 +a^2 - 4ab)$ for some integer $\ell$.  Since $a$ and $b$ are relatively prime, it follows that $\ell$ must be a multiple of $b$.  After canceling this factor of $b$ we get an equation $-2b = \ell'(2b^2 + a^2 -4ab)$ for some integer $\ell'$ which again can be seen to be a multiple of $b$. After cancellation we get an equation $-2 = \ell''(2b^2+a^2-4ab)$ for some integer $\ell''$.  After dividing by $\ell''$ we obtain 
$$2b^2 + a^2 - 4ab = \frac{-2}{\ell''}.$$
The left hand side is an odd integer, which thus must be $\pm 1$.  We have shown that $2b^2 + a^2- 4b = \pm 1$.  Note that this is precisely the condition that $z_0 - y_0 = \pm 1$. 

In the other case we get

$$\frac{xz}{(y-x)(z-y)} = \frac{(b^2-Ab)(b(b-2A))}{(A(2b-2A) -b^2+Ab)(b(b-2A)-A(2b-2A))}$$
$$ = \frac{-b^{2}}{2A^{2}-4Ab+b^{2}}.$$

A similar analysis, in this case using that $b$ is odd we see that $2A^2 -4Ab +b^2 =\pm 1$ and thus that $z_0-y_0 =\pm 1.$  In either case, since $D$ is a degree sequence, $z_0 > y_0$ and the result follows. 
\end{proof}
\begin{Remark}
	We note that this result frankly arises as a sort of miracle.  Were this not the case, we would have a rather messy endeavor to try and clear denominators.  For instance in codimension $4$, we have found the following examples of $D$ with $B_1(D) =2$.  Note that it is {\bf not} always the case that $D_c - D_{c-1} = 1$.  At this time we do not know whether to expect infinitely many such $D$ as in the $c=3$ case.  The difficulty arises because when one tries to clear denominators in the analogous way we just did, there are four numbers $a_1, b_1, a_2, b_2$ representing the numerators and denominators of the ``slopes'' used in the parametrization. One cannot assume that e.g. $a_1$ and $b_2$ are relatively prime, and the analysis becomes quite messy.  Further, even if one were able to do this, the next step would then be to find integer solutions to a cubic diophantine equation, for which we do not know how to proceed.  We hope experts from adjacent fields in number theory and arithmetic geometry may have insights on this behavior.  
$$\begin{array}{|c||c|}\hline 
D  & B(D)  \\ \hline 	
	(0, 1, 4, 5, 6)&(1, \boxed{2}, 5, 6, 23)\\ \hline 
(0,3,13,15,16)&(1,\boxed{2},12,26,15)\\ \hline 
(0, 9, 42, 44, 45) & (1, \boxed{2}, 90, 243, 1543)\\ \hline  
(0,20,95,96,100) & (1, \boxed{2}, 512, 625, 114)  \\ \hline 
(0, 35, 168, 170, 171) & (1, \boxed{2}, 1275, 3724, 2450)\\  \hline 
\end{array}$$

\end{Remark}

We are now able to prove our main result. 
\begin{Theorem}\label{inf family with recursion}
The set of degree sequences $D$ of codimension $3$ with with $B_1(D) = 2$ is an infinite family. They each have the form $\{0, x_n, y_n, y_n+1\}$ for $n\geq 1$ where $x_n,\ y_n$ are given recursively by the following formula: 
$$\begin{bmatrix} x_{n+1}\\ y_{n+1} 	
\end{bmatrix} = 
	\begin{bmatrix}
		-1 & 2 \\ -4 & 7 
	\end{bmatrix}
	\begin{bmatrix}
	x_{n} \\ y_{n} 	
	\end{bmatrix}  + \begin{bmatrix}
		1 \\ 3
	\end{bmatrix}, \hskip .3in
	\begin{bmatrix} x_{0}\\ y_{0} 	
\end{bmatrix} =\begin{bmatrix}
0 \\ 0 	
\end{bmatrix}.
$$
\end{Theorem}

\begin{proof}
Suppose that $D = \{0, x,y,z\}$ is a degree sequence with $B_1(D) = 2$. By Proposition \ref{prop:Lis1ifB1istwo} we know that 
$B_1(D) = \pi(D)$ so we have the following equation:
$$2(y-x)(z-x) = yz.$$
By Theorem \ref{Thm:AllarePell} we know that $z = y+1$.   Making this substitution and rewriting, we get
$$y^{2}+y-4xy = -2x^2+2x.$$
After completing the square, one obtains: 
$$(y-2x)^2 +y = 2x^2+2x,$$ 
$$(y-2x)^2 +(y-2x) = 2x^2.$$
Letting $w = y-2x$ and completing the square once again we get
$$(w+\frac12)^2 = 2x^2+\frac14,$$
$$(2w+1)^2 = 2(2x)^2 + 1.$$
Letting $v = 2w+1$, $u = 2x$ gives us a Pell's equation 
$$v^2 = 2u^2 + 1.$$
The infinitely many solutions to this equation $v_n, u_n$
are given by equating coefficients: 
$$(3 + 2\sqrt2)^n = v_n + u_n \sqrt 2.$$
This yields the recurrence relation that 
\begin{eqnarray*}
v_{n+1}+ u_{n+1}\sqrt2 &=& (3+2\sqrt2)(v_n + u_n \sqrt 2)\\
& = & (3v_n + 4u_n) + (2v_n + 3u_n)\sqrt2. 
\end{eqnarray*}
So that
\begin{equation}\label{eqn:matrix1}
\begin{bmatrix} u_{n+1}\\ v_{n+1} 	
\end{bmatrix} = 
	\begin{bmatrix}
		3 & 2 \\ 4 & 3 
	\end{bmatrix}
	\begin{bmatrix}
	u_{n} \\ v_{n} 	
	\end{bmatrix}.
	\end{equation}
Noting that $x = u/2$ and $y=  u +v/2- 1/2$ we let $A = \begin{bmatrix}
	1/2 & 0 \\ 1 & 1/2
\end{bmatrix}$ and then we note that 
$$A\begin{bmatrix}
u_n \\ v_n 	
\end{bmatrix} + \begin{bmatrix} 0 \\-1/2 
 	\end{bmatrix} = \begin{bmatrix}
x_n \\ y_n 	
\end{bmatrix}.
$$
Now starting with Equation \ref{eqn:matrix1} we obtain: 
$$ \begin{bmatrix} x_{n+1}\\ y_{n+1} 	
\end{bmatrix}=
A\begin{bmatrix} u_{n+1}\\ v_{n+1} 	
\end{bmatrix} +  \begin{bmatrix} 0 \\-1/2
 	\end{bmatrix} = 
	A\begin{bmatrix}
		3 & 2 \\ 4 & 3 
	\end{bmatrix}\begin{bmatrix} u_{n}\\ v_{n} 	
\end{bmatrix} + \begin{bmatrix} 0 \\-1/2
 	\end{bmatrix}$$

$$ = 
	A\begin{bmatrix}
		3 & 2 \\ 4 & 3 
	\end{bmatrix}A^{-1}\left(A\begin{bmatrix} u_{n}\\ v_{n} 	
\end{bmatrix} + \begin{bmatrix} 0 \\ -1/2 	
\end{bmatrix} - \begin{bmatrix} 0 \\ -1/2 	
\end{bmatrix}\right) + \begin{bmatrix} 0 \\-1/2
 	\end{bmatrix}$$

$$ = 
	A\begin{bmatrix}
		3 & 2 \\ 4 & 3 
	\end{bmatrix}A^{-1}\left(\begin{bmatrix} x_{n}\\ y_{n} 	
\end{bmatrix} - \begin{bmatrix} 0 \\ -1/2 	
\end{bmatrix}\right) + \begin{bmatrix} 0 \\-1/2
 	\end{bmatrix}$$

$$= 
	\begin{bmatrix}
	-1 & 2 \\ -4 & 7 
	\end{bmatrix}
	\begin{bmatrix}
	y_{n} \\ x_{n} 	
	\end{bmatrix}
 + \begin{bmatrix} 1 \\ 3
 	\end{bmatrix}. 
	$$
We have thus shown that \emph{if} $B_1(D)= 2$ then $D$ must be of the desired form.  However, it remains to show that $D$ of this form indeed have $\pi(D) = B(D)$.  We prove this in the following Lemma.
\end{proof}
\begin{Lemma}\label{lemma with Bs}
Fix $n \geq 1$ and set $D = \{0, x_n, y_n, y_n+1\}$ with the notation from the previous theorem, then $B(D) = \pi(D) = \{1, 2, y_n - 2x_n+1, y_n - 2x_n\}.$
\end{Lemma}
\begin{proof}
For simplicity, we will drop the subscripts from $x,y$.  In the proof of the previous Theorem we have already shown that $\pi_1(D) = 2$ or that 
\begin{equation}\label{eqn:pi1is2}
y^2 +y - 4xy = -2x^2 + 2x.	
\end{equation}

It remains to show that $\pi_2(D), \pi_3(D) \in \ZZ$.  Since
$$1 - \pi_1(D) + \pi_2(D) - \pi_3(D) = 0$$
the result will follow once we prove that $\pi_2(D) = y - 2x + 1$, or that
$$\frac{x(y+1)}{(y-x)\cdot 1} = y - 2x + 1.$$
Which is equivalent to Equation \ref{eqn:pi1is2}.
\end{proof}

\subsection{Extension to Higher Codimension}\label{subsection3.1}

We conclude this section with a brief discussion of what can (and cannot) be extended to arbitrary codimension.  In general if one wants to search for degree sequences $D$ with $B_i(D) = \alpha$, and $B_0(D) = L$ then one will solve for integer solutions to some diophantine equation coming from the Herzog-K\"uhl equations. The good news is that finding the rational solutions can be done using the same technique that we did at the outset of this section. 

\begin{Theorem}\label{thm:generalprojectionproof}
Suppose that $D = \{0, d_1, d_2, \ldots, d_c\}$. Fix an index $i$.  The condition that $B_i(D) = \alpha$, and $B_0(D) = L$ for some positive integers $\alpha, L$ is equivalent to finding integer solutions $x_1 < x_2, \ldots< x_c$ to the equation:
$$L x_1\cdots \widehat{x_i}\cdots x_c = \alpha (x_i - x_1)\cdots (x_i- x_{i-1})\widehat{(x_i-x_i)}(x_{i+1}-x_i)\cdots (x_c - x_i).$$
The integer solutions $(z_1, \ldots, z_c)$ are all multiples of a rational solution of the form $(q_1, \ldots, q_{c-1}, 1)$.  The solutions to this equation can be parametrized by looking at lines with rational slope through the origin $(0,\ldots, 0, 1)$. Each such line meets the variety defined by the equation at in the origin of degree $c-2$.
\end{Theorem}
\begin{proof}
	The proof is straightforward.  Let $S$ be the hypersurface defined by the equation in the statement of the theorem (with $x_c = 1$).  Parametrize a rational line $r(t)$ through the origin in $\QQ^{c-1}$.  By this we mean a parametrization $$r(t) = (m_1t, \ldots, m_{c-1}t)$$ for some rational numbers $m_i$.  Any point with rational coordinates on $S$ must lie on some such line.  To see the converse, suppose that $P$ is a point that lies on some rational line parametrized by $r(t)$.  Then $P$ must be a solution to an equation of degree $c-1$ in $t$ that is obtained by substituting the coordinates of $r(t)$ into the equation for $S$, which is: 
	$$L m_1t\cdots \widehat{m_it}\cdots m_{c-1}t = $$
	$$\alpha (m_it - m_1t)\cdots (m_it- m_{i-1}t)\widehat{(m_it-m_it)}(m_{i+1}t-m_it)\cdots(m_{c-1}t-m_it)(1 - m_it).$$
One sees that $t^{c-2}$ is a factor of both sides, and thus the origin is a root of multiplicity $c-2$. After dividing through by this factor, solving for $t$ (and thus finding the coordinates of $P$) simply requires solving a linear equation.  Thus the coordinates of $P$ must be rational. 
\end{proof}

\begin{Remark}
Recall that when we did these calculations for $c=3$ we obtained rational parametrizations that depended on a slope $m = a/b$, so there were two parameters $a$ and $b$:
$$\left(1+\frac{a}{2(b-a)}, \frac{a(2b-a)}{2b(b-a)},1\right).$$ 
The content of Theorem \ref{thm:generalprojectionproof} says that this can be done in any codimension.  However, if one is interested in {\bf integer} solutions, one must then clear denominators.  In the codimension $3$ case we were quite lucky that the fractions were essentially in lowest terms so that clearing fractions could be done just by multiplying by the denominators.  This was because of the helpful assumption that we could assume $\gcd(a,b) =  1$, since we may as well take the fraction $m = a/b$ to be in lowest terms.  In higher codimension, we will have many fractions $m_i = a_i/b_i$ and while we can assume that $\gcd(a_i,b_i) = 1$, in general it may happen that for instance $a_3$ and $b_2$ share a factor.  The following example illustrates what can happen. 
\end{Remark}

\begin{Example}
If $c=4$ and we want to find those degree sequences with $B_2(D) = 4$ and $B_0(D) = 1$ then we will solve the equation 
$$yzw = 2(y-x)(z-x)(w-x).$$
We set $w = 1$ and  parametrizing with $r(t) = (m_1t,m_2t,m_3t)$.  Without loss of generality, since we are interested in nonzero solutions, we can reparametrize by setting any $m_i = 1$.  If we let $m_i = a_i/b_i$ then we get:   
$$(\frac{a_2a_3}{b_2b_3})t^2 = 2(\frac{a_2}{b_2}t-t)(\frac{a_3}{b_3}t-t)(1-t)$$
$$a_2a_3 = 2(a_2 - b_2)(a_3 - b_3)(1-t)$$
$$t = 1 - \frac{a_2a_3}{2(a_2 - b_2)(a_3 - b_3)}.$$
Then our solutions are 
$$(x,y,z,w) = 
\left(t, \frac{a_2}{b_2}\cdot t, \frac{a_3}{b_3}\cdot t, 1\right)=$$
$$
\left(
1 - \frac{a_2a_3}{2(a_2 - b_2)(a_3 - b_3)},
\frac{a_2}{b_2} - \frac{a_2^2a_3}{2b_2(a_2 - b_2)(a_3 - b_3)},
\frac{a_3}{b_3} - \frac{a_2a_3^2}{2b_3(a_2 - b_2)(a_3 - b_3)},
1\right) = 
$$
$$
\left(
1 - \frac{a_2a_3}{2(a_2 - b_2)(a_3 - b_3)},
\frac{a_2(a_2a_3-2a_3b_2-2a_2b_3+2b_2b
      _3)}{2b_2(a_2 - b_2)(a_3 - b_3)},
\frac{a_3(a_{2}a_3-2a_{3}b_{2}-2\,a_{2}b_{3}+2b_{2}b_{3})}{2b_3(a_2 - b_2)(a_3 - b_3)},
1\right).
$$
Notice that if $m_2 = 13/3$ and $m_3 = 5/1$ then the first fraction is not in lowest terms because $a_2 - b_2 = 10$ has a factor in common with $a_3$. This is the type of issue that one would have to address to find all integer solutions. 
\end{Example}

\section{Codimension 3 Degree Sequences with $\sum B_i < 12$}
We now turn to study those degree sequences $D$ of codimension $3$ that violate Main Conjecture (2), that is $\sum B(D) < 12$. We will do everything by hand and use the results from the previous section as well.  In Section \ref{Sec:Gor} we will show how such methods might be extended to other cases.

\begin{Lemma}\label{lemma:17}
	For any degree sequence $D$ of codimension $3$, 
	$$\sum B(D) \geq 2B_0(D) + 2B_3(D) + 2.$$
In particular, 
	\begin{itemize}
	\item If $B_0(D) \geq 4$ then $\sum B(D) \geq 12$;
	\item If $B_0(D) = 3$ then $\sum B(D) \geq 12$ unless $B_3(D) = 1$;
	\item If $B_0(D) = 2$ then $\sum B(D) \geq 12$ unless $B(D) = \{2,3,3,2\}$  or $B_3(D) = 1$.
	\end{itemize}
Thus apart from $B(D) = \{2,3,3,2\}$, any degree sequence with $\sum B(D) < 12$ must satisfy $B_0(D) = 1$ or $B_3(D) = 1$.\end{Lemma}
\begin{proof}
By Proposition \ref{prop:Lis1ifB1istwo} applied to $D$ and its dual we have that $B_1(D) \geq B_0(D) + 1$ and $B_2(D) \geq B_3(D)+1$.  The results follow. 
\end{proof}

\begin{Theorem}\label{theorem for equal to 10} Let $D$ be degree sequence of codimension $3$. The only degree sequences $D$ with $\sum B(D) < 12$ are in the table below: 

$$\begin{array}{|c|c|c|}\hline 
D & B(D) & \mbox{sum} \\ \hline
\{0,1,3,4 \} & \{1,2,2,1 \}  & 6\\	 \hline
\{0,1,2,3 \} & \{1,3,3,1 \}  & 8\\ \hline
\{0,1,5,6 \} & \{2,3,3,2 \}  & 10\\ \hline
\{0,3,5,8\} & \{1,4,4,1 \}  & 10 \\ \hline
\end{array}$$
\end{Theorem}
\begin{proof}
Dualizing if necessary, we will assume that $B_0(D) \leq B_3(D)$.  In other words, we put the smaller number first.  By the previous lemma, we may assume $B_0(D) \leq 2$ and if $B_0(D) = 2$ then $B(D) = \{2,3,3,2\}$.  

Suppose that $B_0(D) = 1$, i.e. that $B(D) = \{1, a,b,c\}$ for some $a,b,c$ with $c \leq 3$ by Lemma \ref{lemma:17}. Since the alternating sum of the entries of $B(D)$ is zero, we must have that 
$$B(D) = \{ 1, a, a+c-1, c\}.$$
Analyzing cases for $c \in \{1,2,3\}$ yields the following possibility for $B(D)$ with $\sum B(D) < 12:$
$$\{1,2,2,1\}, \{1,3,3,1\}, \{1,4,4,1\}, \{1,2,3,2\},\{1,3,4,2\},\{1,2,4,3\}$$
Note that Theorem \ref{Thm:AllarePell} characterizes those $D$ with $B_1(D) = 2$, and we note that $\{1,2,3,2\},\{1,2,4,3\}$ do not occur, and that $\{1,2,2,1\}$ occurs if and only if $D = \{0,1,3,4\}$. 

Thus up to duality, the only sequences $B(D)$ we need to study are $\{1,3,3,1\}$, $\{1,4,4,1\}$, $\{1,3,4,2\}$, and $\{2,3,3,2\}$.  We proceed in cases. In all of these we will write $D = \{0, x, y, z\}$ and suppose that $D$ is non-degenerate, i.e. $\gcd(x,y,z) =1$. 

\medskip 

\noindent {\bf Case 1: The sequence $\{1,3,4,2\}$ does not occur as $B(D)$:}

Suppose $B(D) = \{1,3,4,2\}$. Then we have the following equations:
$$yz = 3(y-x)(z-x),$$ 
$$xz = 4(y-x)(z-y),$$
$$xy = 2(z-x)(z-y).$$
The middle equation implies that $xz$ is divisible by $4$. 

Suppose that $x$ and $z$ were both even. Then $y$ must be odd since $D$ is non-degenerate.  But this means that $(y-x)(z-y)$ is odd and thus that $(xz)/4$ is odd, so that $x/2$ and $z/2$ are both odd.  However, in the last equation the factor $2(z-x)$ will be a multiple of $4$, which is impossible since the left hand side is a multiple of $2$ but not a multiple of $4$.  

Thus we must have that either $x$ is a multiple of $4$ and $z$ is odd, or vice versa.  In particular, we have $z-x$ is odd.  
Dividing the second equation by the third and simplifying yields
$$z(z-x) = 2y(y-x).$$
Thus $z$ must be even and $x$ odd.  Suppose that $z = 2^qk$ for $k$ odd. 
The top equation then forces $y-x$ to be a multiple of $2^q$, say $y-x=2^q\ell$. Then the middle equation becomes:  
$$2^qkx = 4\cdot 2^q \ell (z-y)$$
which is a contradiction since this would imply that $kx$ (an odd number) is divisible by $4$.  Hence $\{1,3,4,2\}$ does not occur as $B(D)$.

\medskip

\medskip 

\noindent {\bf Case 2: $D$ with $B(D) =\{1,a,a,1\}$:} 

Suppose $B(D) = \{1,a,a,1\}$ for $a \in \{3,4\}$. Then we have the following equations:
$$yz = a(y-x)(z-x),$$ 
$$xz = a(y-x)(z-y),$$
$$xy =  (z-x)(z-y).$$
The last equation implies that $z(z-x-y) = 0,$ so $z = x+y$.  The other two equations then both become equivalent to $x+y=a(y-x)$ which implies that $y = (a+1)x/(a-1)$.  

If $a=3$ then this implies $D$ must be of the form $\{0, x, 2x, 3x\}$ and non-degeneracy forces $D = \{0,1,2,3\}$.

If $a=4$ then this implies $D$ must be of the form $\{0, x, \frac53 x, \frac83 x\}$ and non-degeneracy forces $D = \{0,3,5,8\}$.

\medskip 

\noindent {\bf Case 3: $D$ with $B(D) = \{2,3,3,2\}$:}

Suppose $B(D) = \{2,3,3,2\}$. Then we have the following equations:
$$2yz = 3(y-x)(z-x),$$ 
$$2xz = 3(y-x)(z-y),$$
$$2xy = 2(z-x)(z-y).$$
Note the factor of 2 on the left because $B_0 = 2$. 

The last equation implies that 
$2(z^2-xz-yz) = 0$, so $z = x+y$.  Substituting this into the other equations we get $2(x+y) = 3(y-x)$ or $y = 5x$ so $D = \{0,x,5x,6x\}$ and thus since $D$ is nondegenerate, $D = \{0,1,5,6\}$.

\medskip 

These calculations conclude the proof that the table in the statement of the Theorem is complete.
\end{proof}
\begin{proof}[Proof of Theorem \ref{thm1}:]
We close by completing the proof of Theorem \ref{thm1}. All that remains is to show that for each arithmetic obstruction we have found where $B(D)$ violates a conjectured bound that in the equation $\beta(M) = LB(D)$ we must have $L \geq 2$.  

Note that the $B(D)$ violating conjecture $1$ all had the form $\{1, 2, B_2, B_3\}$ or $\{B_0, B_1, 2, 1\}$.  Since we are in codimension 3, the Krull altitude theorem rules out the first sequence as a possible betti sequence of a module $M$.  The same argument applied to the dual of $M$ rules out the second. 

Now consider the sequences $B(D)$ from Theorem \ref{theorem for equal to 10}.  The sequence $\{1,2,2,1\}$ is not the betti sequence of a module by the Krull altitude theorem.  The second is the betti sequence of a complete intersection, and so it is allowed by Conjecture 2.  The third violates the general fact that $\beta_1(M) \geq \beta_0(M) + (c-1)$ (see e.g. Theorem 2.1 in \cite{BERemarks}).  And finally, the sequence $\{1,4,4,1\}$ is not the betti sequence of a module because an almost complete intersection is never Gorenstein. Thus $L\geq 2$ in all of these cases. 
\end{proof}

\section{Gorenstein Rings in Codimension 5 and 6}\label{Sec:Gor}
In this section will we prove Theorem \ref{Thm:GorCodim5}.
Heretofore we have been interested in finding a complete solution to Program \ref{Program:1}.  In other words, we wanted to find all arithmetic obstructions.  However, if one is interested in proving new cases of the Main Conjectures, we may as well throw into the mix the necessary conditions imposed by the Krull altitude theorem, Kunz's Theorem, and others.  With this approach we the next open case of the conjecture is just barely within reach.  

\begin{proof}[Proof of Theorem \ref{Thm:GorCodim5}]
We start by proving the claims in codimension $5$: 

Suppose that $\beta = \{1,a,b,b,a,1\}$ is the betti sequence of a Gorenstein ring $M = R/I$, and assume that $M$ is not isomorphic to a complete intersection. To prove Main Conjecture 1, we need to show that $a \geq 5$ and $b\geq 10$.  The first is simply the Krull altitude theorem.  Since almost complete intersections are never Gorenstein by a Theorem of Kunz \cite{Kunz}, we may assume that $a\geq 7$.  
Using the Syzygy Theorem of Evans-Griffith \cite{EG}, we see that the rank of the 3rd map in the resolution must be at least 3.  Since the second map evidently has rank $a-1$ we deduce that $b \geq 3 + (a-1)$ so $b \geq 2+a$.  Thus if $a \geq 8$ the Main Conjecture 1 holds. If $a = 7$, then the only betti sequence to rule out is $\{1, 7, 9, 9, 7,1\}$ which will be ruled out below as it would also violate Main Conjecture 2. We now focus on finding all choices of $a,b$ that would violate this latter conjecture. 

Since $b\geq 2+a$, the sum of the betti numbers is at least $6+4a$.  So we only need to consider the cases when $a\in \{7,8,9,10\}$.  A complete list of possible $(a,b)$ is below:
	$$(7,9), (7,10),(7,11),(7,12),(7,13),(7,14),(7,15),(8,10),$$ 
	$$(8,11), (8,12), (8,13),(8,14),(9,11),(9,12),(9,13),(10,12)$$
		
	Since $R$ is Gorenstein and pure its degree sequence is of the form $\{0,x,y,z,z+y-x, y+z\}$ for $0< x< y< z.$  Since $\beta_0(R) = 1$ we must have that $B(D) = \beta(D)$ and thus we have 
\begin{eqnarray*}
	yz(z+y-x)(y+z) &=& a\cdot (y-x)(z-x)(z+y-2x)(y+z-x) \\
	xz(z+y-x)(y+z) &=& b\cdot (y-x)(z-y)(z-x)z
\end{eqnarray*}	
by the Herzog-K\"uhl equations.
Notice the common factor that can be canceled from each equation.  For each $(a,b)$ pair, these equations define an affine variety $X_{a,b}$ in $\QQ^3$ and we want to determine whether there is an integer point $(x,y,z)$ on $X_{a,b}$ with $0<x<y<z$.  To do this we use Macaulay2 to find the irreducible components for $X_{a,b}$.   Namely we use the \texttt{decompose} command in Macaulay2 to find the components.  It turns out that each variety has $9$ components: $8$ lines and one cubic plane curve. The lines all are defined by equations either setting some coordinate to zero, or setting two coordinates equal to one another.  Those cannot contain a point $(x,y,z)$ with $0<x<y<z$.  The component with a cubic plane curve is defined by a linear form $f$ and a cubic in $y$ and $z$. These components are listed below:
$$\begin{array}{l}
\texttt{ideal}{}\left(14x-15y+3z,20y^{3}-75y^{2}z+204yz^{2}-85z^{3}\right)\\
\texttt{ideal}{}\left(7x-8y+2z,12y^{3}-44y^{2}z+81yz^{2}-33z^{3}\right)\\
\texttt{ideal}{}\left(14x-17y+5z,85y^{3}-306y^{2}z+475yz^{2}-190z^{3}\right)\\
\texttt{ideal}{}\left(7x-9y+3z,33y^{3}-117y^{2}z+165yz^{2}-65z^{3}\right)\\
\texttt{ideal}{}\left(14x-19y+7z,190y^{3}-665y^{2}z+882yz^{2}-343z^{3}\right)\\
\texttt{ideal}{}\left(7x-10y+4z,65y^{3}-225y^{2}z+286yz^{2}-110z^{3}\right)\\
\texttt{ideal}{}\left(14x-21y+9z,343y^{3}-1176y^{2}z+1449yz^{2}-552z^{3}\right)\\
\texttt{ideal}{}\left(16x-17y+3z,51y^{3}-187y^{2}z+513yz^{2}-209z^{3}\right)\\
\texttt{ideal}{}\left(8x-9y+2z,15y^{3}-54y^{2}z+100yz^{2}-40z^{3}\right)\\
\texttt{ideal}{}\left(16x-19y+5z,209y^{3}-741y^{2}z+1155yz^{2}-455z^{3}\right)\\
\texttt{ideal}{}\left(8x-10y+3z,40y^{3}-140y^{2}z+198yz^{2}-77z^{3}\right)\\
\texttt{ideal}{}\left(16x-21y+7z,65y^{3}-225y^{2}z+299yz^{2}-115z^{3}\right)\\
\texttt{ideal}{}\left(18x-19y+3z,95y^{3}-342y^{2}z+945yz^{2}-378z^{3}\right)\\
\texttt{ideal}{}\left(9x-10y+2z,55y^{3}-195y^{2}z+363yz^{2}-143z^{3}\right)\\
\texttt{ideal}{}\left(18x-21y+5z,378y^{3}-1323y^{2}z+2070yz^{2}-805z^{3}\right)\\
\texttt{ideal}{}\left(20x-21y+3z,77y^{3}-273y^{2}z+759yz^{2}-299z^{3}\right)
\end{array}$$

We claim that none of these components have a point $(x,y,z)$ with $z>0$.  If they did, then because the cubics are homogeneous, then after setting $z=1$ the equations would have a rational solution.  One can use the rational root test to complete the argument. 

The proof in codimension $6$ proceeds in a similar way.  We will highlight the differences. The betti sequences to consider here will have the form $\beta = \{1,a,b,2b-2a+2,b,a,1\}$. 
Notice that there are infinitely many pairs $(a,b)$ where $\beta_3 < 20$.  This means that verifying Main Conjecture (1) for $\beta_3$ cannot be done in a with a finite calculation.   For $\beta_1$, note that by the Total Rank Theorem, if $M$ is not isomorphic to a complete intersection, then the sum of the betti numbers must be greater than $2^c = 64$.  Thus we have the $4+4b > 64$ and thus $\beta_2 = b \geq 16 > 15 = {c \choose 2}$.  Note that we have not used at all that $M$ is a pure module, (see Corollary \ref{cor:generalGOR}.)

For Main Conjecture (2), we need to show that the sum of the betti numbers, $4+4b$ is at least $2^6+2^5$, that is that $b \geq 23$. By the previous paragraph and the Syzygy Theorem, we have $b \geq \min(a + 3, 16)$ and $a\geq 8$ by Kunz' theorem. There are then $63$ pairs $(a,b)$ to check.   Again, since $M = R/I$ is a cyclic module, $B(D) = \beta(D)$, and since $R/I$ is Gorenstein we know that $D$ has the form $\{0,x,y,z,2z-y,2z-x,2z\}$ for some integers $0< x< y< z$.  As before we use the Herzog-K\"uhl equations to define varieties $X_{a,b}$. 

All of the 63 varieties $X_{a,b}$ have exactly 5 components containing lines defined by the vanishing of a variable or difference of two variables, and thus do not correspond to valid degree sequences.  After eliminating these components, all but one of the corresponding varieties $X_{a,b}$ has only one remaining component, which contains a degree 4 polynomial in two variables $y,z$.  Applying the rational roots test shows that none of these components contain nonzero integer points.  

The lone exception is $X_{16,20}$ (with potential $\beta = \{1,16, 20, 10, 20, 16, 1\}$) which has $3$ non-linear components: 

$$\{\texttt{ideal}(z,x-y), \texttt{ideal}(z,x+y), \texttt{ideal}(y-2z,x-z), \texttt{ideal}(y-z,x-2z), $$ 
$$\texttt{ideal}(2y-3z,8x^2-16xz+5z^2), \texttt{ideal}(2y-z,8x^2-16xz+5z^2), $$ 
$$\texttt{ideal}(4y^2-8yz+7z^2,8x^2-16xz+15z^2)\}$$
One can still apply the rational roots test to see that these components contain no nontrivial integer points.
\end{proof}
The proof of the previous theorem shows the following more general result, without any assumption about purity:
\begin{Corollary}\label{cor:generalGOR}
	If $R/I$ is any finitely-generated graded Gorenstein algebra of codimension $6$, that is not isomorphic to a complete intersection then $\beta_2(R/I) \geq 16$.
\end{Corollary}

\begin{Remark}
We found these calculations rather miraculous in how quickly they settled this case.  However, the brevity of this argument belies how subtle, too, this approach can be.   For instance, our search excluded the pair $(a,b)=(10,11)$ because this would violate the Syzygy Theorem.  Nonetheless, however, this actually {\bf does} occur as a $B$-sequence: 
$$B(\{0,3,4,10,11,14\}) = \{1,10,11,11,10,1\}.$$
Sure enough, when one decomposes the corresponding ideal, one gets a component 
$$\texttt{ideal}{} (5y - 2z, 10x - 3z)$$
from which we can recover the degree sequence, with $(x,y,z)=(3,4,10)$.
\end{Remark}

\noindent {\bf How far can we go?} Since Main Conjecture 2 is ``finite'' in the sense that there are only finitely many betti sequences to check for a given $c$, one might hope that this technique will continue for $c = 7$ and beyond.  What happens? Starting in $c=7$ there will be new parameters in both the betti sequence and also the degree sequence.  The primary decomposition step then becomes much more difficult.  

Further, at some point we know that this process {\bf cannot} work, because there will be self-dual degree sequences $D$ with $B(D)$ that violates Main Conjecture 2.  The smallest example we found is the one in Question \ref{quest:gor}.  It is in codimension $15$ and has 
$$D = \{0,2,3,4,5,6,7,9,11,13,14,15,16,17,18,20\},$$ $$B = \{1,85,630,2295,4998,6630,4590,2210,2210,4590,6630,4998,2295,630,85,1\}.$$

\section*{Acknowledgements}

We thank Daniel Erman, Karthik Ganapathy, Elo\'isa Grifo, Srikanth Iyengar, Claudiu Raicu and Mark Walker for helpful conversations. Support for this research was provided by an AMS-Simons Research Enhancement Grant for Primarily Undergraduate Institution Faculty. We are grateful to the University of San Diego, including the support from the Office of Undergraduate Research, and the Fletcher Jones Endowment for Applied Mathematics.
\bibliographystyle{plain} 
\bibliography{References}
\end{document}